\documentclass[11pt]{article}

\usepackage{amsmath,amssymb,amsfonts,amsthm}

\usepackage{tikz}
\usepackage{graphicx}
\usepackage{placeins}
\usepackage{pgfplots}
\usetikzlibrary{automata}
\usetikzlibrary{arrows}
\usetikzlibrary{positioning,calc}
\usetikzlibrary{graphs}
\usetikzlibrary{graphs.standard}
\usetikzlibrary{arrows,decorations.markings}
\usepackage{tkz-graph}
\usetikzlibrary{chains,fit,shapes}
\usetikzlibrary{calc}
\tikzset{every loop/.style={min distance=10mm,looseness=10}}
\tikzset{every state/.style={minimum size=2mm}}
\tikzstyle arrowstyle=[scale=1]
\tikzstyle directed=[postaction={decorate,decoration={markings,mark=at position 0.7 with {\arrow[arrowstyle]{stealth};}}}]

\newtheorem{theorem}{Theorem}

\newtheorem{lemma}[theorem]{Lemma}

\newtheorem{corollary}[theorem]{Corollary}

\title{On the word-representability of $K_m$-$K_n$  graphs}

\author{Herman~Z.~Q.~Chen\footnote{School of Mathematical Sciences; Chongqing Key Lab of Cognitive Intelligence and Intelligent Finance, Chongqing Normal University, Chongqing 400047, China.  {\bf Email:} chern@cqnu.edu.cn},\ \ Humaira Hameed\footnote{Department of Mathematics and Statistics, University of Strathclyde, 26 Richmond Street, Glasgow G1, 1XH, United Kingdom. 
{\bf Email:} humaira.hameed@strath.ac.uk.}\ \ and Sergey Kitaev\footnote{Department of Mathematics and Statistics, University of Strathclyde, 26 Richmond Street, Glasgow G1, 1XH, United Kingdom. 
{\bf Email:} sergey.kitaev@strath.ac.uk.}}

\begin{document}
	\maketitle
	
\begin{abstract}
Word-representable graphs are a class of graphs that can be represented by words, where edges and non-edges are determined by the alternation of letters in those words. Several papers in the literature have explored the word-representability of split graphs, in which the vertices can be partitioned into a clique and an independent set. In this paper, we initiate the study of the word-representability of graphs in which the vertices can be partitioned into two cliques. We provide a complete characterization of such word-representable graphs in terms of forbidden subgraphs when one of the cliques has a size of at most four. In particular, if one of the cliques is of size four, we prove that there are seven minimal non-word-representable graphs. \\

\noindent
{\bf Keywords:} word-representable graph; semi-transitive graph; semi-transitive orientation\\

\noindent
{\bf 2020 Mathematics Subject Classification:} 05C62, 20M99\\

\end{abstract}

\section{Introduction}

Two letters $x$ and $y$ alternate in a word $w$ if after deleting in $w$ all letters but the copies of $x$ and $y$ we either obtain a word $xyxy\cdots$ or a word $yxyx\cdots$ (of even or odd length).  A graph $G=(V,E)$ is {\em word-representable} if there exists a word $w$
over the alphabet $V$ such that letters $x$ and $y$, $x\neq y$, alternate in $w$ if and only if $xy\in E$; each letter in $V$ must appear in $w$. The unique minimal (by the number of vertices) non-word-representable graph on 6 vertices is the wheel graph $W_5$ (which is the cycle graph $B_5$ with an all adjacent vertex added), while there are 25 non-word-representable graphs on 7 vertices and they can be found in \cite{KS}. 

The introduction of word-representable graphs was influenced by alternation word digraphs used in \cite{KitSeif}, which served as a tool to study the celebrated Perkins semigroup, {\bf $B^1_2$}, a central concept in semigroup theory since 1960, particularly as a source of examples and counterexamples.

An orientation of a graph is {\em semi-transitive} if it is acyclic (there are no directed cycles), and for any directed path $v_0 \rightarrow v_1 \rightarrow \cdots \rightarrow v_k$, either there is no edge between $v_0$ and $v_k$, or $v_i \rightarrow v_j$ is an edge for all $0 \leq i < j \leq k$. An undirected graph is {\em semi-transitive} if it admits a semi-transitive orientation. A shortcut in an acyclically oriented graph $G$ is an induced {\em non-transitive} subgraph $G'$ on vertices $\{v_0, v_1, \ldots, v_k\}$, $k \geq 3$, containing $v_0 \rightarrow v_1 \rightarrow \cdots \rightarrow v_k$ and $v_0 \rightarrow v_k$ (that is, $v_i \rightarrow v_j$ is not present in $G'$ for at least one pair of $i$ and $j$ with $0 \leq i < j \leq k$). The edge $v_0 \rightarrow v_k$ is called {\em shortcutting edge}.

A fundamental result in the area of word-representable graphs is the following theorem.

\begin{theorem}[\cite{HKP16}]\label{fundamental} A graph is word-representable if and only if it admits a semi-transitive orientation. \end{theorem} 

\begin{corollary}[\cite{HKP16}]\label{3-col} Any $3$-colourable graph is word-representable. \end{corollary} 

There is an extensive line of research in the literature on word-representable graphs~\cite{KL}. Most relevant to our work is a series of papers~\cite{CKS21,I2021,IK2021,KLMW17,KP1} that explore the word-representability of split graphs. A split graph is a graph whose vertices can be partitioned into a clique and an independent set~\cite{FH77}. Split graphs appear in the literature in various contexts (see, for example,~\cite{CJKS2020} and references therein).

In this paper, we extend the study of word-representability of split graphs to what we call {\em $K_m$-$K_n$ graphs} by characterizing, in terms of forbidden subgraphs, word-representable $K_m$-$K_n$ graphs for any $1\leq m\leq 4$ and $n\geq 1$. A $K_m$-$K_n$ graph is one that can be split into cliques $K_m$ and $K_n$, where $K_m$ is maximal; that is, there is no vertex in $K_n$ that is connected to every vertex in $K_m$. Let $V(K_m)=\{1,\ldots,m\}$ and $V(K_n)=\{m+1,\ldots,m+n\}$, where for a graph $G$,  $V(G)$ is the set of vertices in $G$. 

Both split graphs and $K_m$-$K_n$ graphs are relevant to the {\em speed of hereditary classes of graphs} and their asymptotic structure, which have been extensively studied in the literature. These concepts were used in \cite{KL1} to find asymptotics for the number of word-representable graphs. To elaborate, it is known that for every 
hereditary class $X$, that is different from the class of all finite graphs,
\begin{equation*}\label{eq:0}
\lim\limits_{n\to\infty}\frac{\log_2 X_n}{\binom{n}{2}}=1-\frac{1}{k(X)},
\end{equation*}
where $k(X)$ is a natural number known as the {\it index} of $X$. To define this notion 
let us denote by ${\cal E}_{i,j}$ the class of graphs whose vertices can be partitioned
into at most $i$ independent sets and $j$ cliques. In particular, ${\cal E}_{p,0}$
is the class of $p$-colourable graphs.
Then $k(X)$ is the largest $k$ such that $X$ contains ${\cal E}_{i,j}$ with $i+j=k$ for some $i$ and $j$. 
This result was obtained independently by Alekseev \cite{Ale92} and Bollob\'{a}s and Thomason  \cite{BT95,BT97}
and is known nowadays as the Alekseev-Bollob\'{a}s-Thomason Theorem (see e.g. \cite{ABT-theorem}).

\subsection{Relevant tools to study word-representability of graphs}
The following theorem is a special case of Theorem 5.4.7 in~\cite{KL}, and it allows the reduction of questions about the word-representability of arbitrarily large graphs within certain graph families to smaller, more manageable graphs.

\begin{theorem}[\cite{KL}]\label{same-neighbourhood}
Suppose $u$ and $v$ are two vertices in a graph $G$ and neighbourhoods of $u$ and $v$ are the same, except $u$ and $v$ are allowed to be connected. Moreover, let $G'$ be obtained from $G$ by removing $u$. Then, $G$ is word-representable if and only if $G'$ is word-representable. 
\end{theorem}

The following two lemmas are instrumental for us in proving non-word-representability of graphs. A vertex in a graph $G$ is a {\em source} (resp., {\em sink}) if all edges incident with $v$ are oriented outwards (resp., towards) it.

\begin{lemma}[\cite{KP}]\label{lemma} Suppose that an undirected graph $G$ has a cycle $C=x_1x_2\cdots x_mx_1$, where $m\geq 4$ and the vertices in $\{x_1,x_2,\ldots,x_m\}$ do not induce a clique in $G$.  If $G$ is oriented semi-transitively, and $m-2$ edges of $C$ are oriented in the same direction (i.e. from $x_i$ to $x_{i+1}$ or vice versa, where the index $m+1:=1$) then the remaining two edges of $C$ are oriented in the opposite direction.\end{lemma}

\begin{lemma}[\cite{KS}]\label{source-thm} Suppose that a graph $G$ is word-representable, and $v$ is a vertex in $G$. Then, there exists a semi-transitive orientation of $G$ where $v$ is a source (or a sink).   
 \end{lemma}

\subsection{Our methodology and organization of the paper}\label{format-subsec}

Proving that a given graph is not word-representable, or equivalently, not semi-transitive, often involves examining all possible extensions of partial orientations and demonstrating that none results in a semi-transitive orientation. Lemmas~\ref{lemma} and~\ref{source-thm} are crucial here because they dramatically reduce the number of orientations to consider. We refer to \cite{KS} for more details on this approach.

By a ``line'' of a proof we mean a sequence of instructions that directs us in orienting a partially oriented graph and necessarily ends with detecting a shortcut showing that this particular orientation branch will not produce a semi-transitive orientation. The idea is that if no branch produces a semi-transitive orientation then the graph is non-semi-transitive. 

Each proof begins with assumptions on orientations of certain edges, and there are four types of instructions:
\begin{itemize}
\item ``B$a\rightarrow b$ (Copy $x$)'' means ``Branch on edge $ab$, orient the edge as $a\rightarrow b$, create a copy of the current version of the graph except orient the edge $ab$ there as $b\rightarrow a$, and call the new copy $x$; leave Copy $x$ aside and continue to follow the instructions''. The instruction B occurs when no application of Lemma~\ref{lemma} is possible in the partially oriented graph.
\item ``MC$x$'' means ``Move to Copy $x$'', where Copy $x$ of the graph in question is a partially oriented version of the graph that was created at some point in the branching process. This instruction is always followed by an oriented edge $a\rightarrow b$ to remind the reader of the directed edge obtained after the application of the branching process. 
\item ``O$a\rightarrow b$(C$abc$)'' means orient the edge $ab$ as $a\rightarrow b$ in the cycle $abc$ to avoid a directed cycle. If instead of a triangle we see a longer cycle, then we deal with an application of Lemma~\ref{lemma} to get a cycle where all but two edges are oriented in one direction, and one of the remaining two edges is oriented in the opposite direction.
\item ``O$a\rightarrow b$ O$c\rightarrow d$ (C$xyz\cdots$)'' means that Lemma~\ref{lemma} is applied to cycle $xyz\cdots$ to create new directed edges, $a\rightarrow b$ and $c\rightarrow d$. 
\end{itemize}

Each line ends with ``S: $xy\cdots z$'' indicating a shortcut with the shortcutting edge $x\rightarrow z$ is obtained. 

We note that for smaller graphs, particularly those with relatively few edges, proofs of non-word representability can be obtained in an automated manner using user-friendly software \cite{S}. Any proof of ``reasonable size'' can be easily verified \cite{KS}.

For a class of graphs in question (with a fixed $K_m$), we consider the graph on $m+2^m-1$ vertices obtained from $K_m$ by adding vertices that are connected to all possible subsets of vertices in $K_m$ of size at most $m-1$, ensuring that no two newly added vertices share the same neighbourhood. Then, we create the clique $K_n$, where  $n=2^m-1$, from the newly added vertices and denote the resulting graph (obtained by adding the new edges and vertices to $K_m$) as $H_m$. By Theorem~\ref{same-neighbourhood}, finding all minimal non-word-representable subgraphs for this class of graphs is equivalent to finding all minimal non-word-representable subgraphs for $H_m$, which can be achieved by sequentially removing the vertices of $H_m$ one by one in all possible ways. If $H_m$ is word-representable then all graphs in the class are word-representable.

To shorten our proofs of Theorems~\ref{thm-K3-Kn} and~\ref{K4-Kn-thm}, we first establish in Lemmas~\ref{lemma-K3-Kn} and~\ref{lemma-K4-Kn} the non-word-representability of all minimal non-word-representable graphs under consideration (as identified through computer experiments). We then use this result to argue that no other minimal non-word-representable graphs exist within the class. The latter is achieved through a branching process in which vertices are removed one by one, as described earlier. A branch terminates when a word-representable graph is obtained. At each step, the non-word-representability of a graph is demonstrated by explicitly identifying a minimal non-word-representable subgraph. Word-representability, on the other hand, can be established either by providing a semi-transitive orientation or by verifying that the graph is not among the 26 known non-word-representable graphs with at most 7 vertices. More commonly, however, this is done using publicly available, user-friendly, and independently developed software~\cite{Glen, S}, which significantly reduces the space required for this paper.

The paper is organized as follows. In Section~\ref{sec2}, we characterize word-representable $K_m$-$K_n$ graphs for $m \in {1, 2, 3}$. In Section~\ref{sec-K4-Kn}, we characterize word-representable $K_4$-$K_n$ graphs. Two of the cases in the proof of Theorem~\ref{K4-Kn-thm} in Section~\ref{sec-K4-Kn}, which together require more than five pages, have been moved to the Appendix. Finally, in Section~\ref{concluding}, we provide concluding remarks.

\section{$K_m$-$K_n$ graphs for $m\in\{1,2,3\}$}\label{sec2}

In what follows, if $u$ is a vertex in a graph $G$ and $G'$ is a subgraph of $G$, then $N_{G'}(u)$ is the set of neighbours of $u$ (the vertices connected to $u$ by an edge) in $G$ that belong to $G'$.

\begin{theorem}
	Any graph $K_m$-$K_n$, where $m\in\{1,2\}$ and $n\ge 1$, is word-representable.
\end{theorem}
\begin{proof}
Let $m=1$. By definition, there is no edge between $K_1$ and $K_n$ (otherwise, $K_n$ would not be maximal), and the word $1123\ldots(n+1)$
	 represents the graph $K_1$-$K_n$.

Let $m=2$. $H_2$, the largest possible graph $K_2$-$K_n$ we need to consider, is on 5 vertices, $V(K_n)=\{3,4,5\}$, $N_{K_2}(3)=\emptyset$, $N_{K_2}(4)=\{1\}$, and $N_{K_2}(5)=\{2\}$, where $K_2$ is the subgraph of $H_3$ given by the vertices in $\{1,2\}$. Since $H_3$ is 3-colourable, by Corollary~\ref{3-col}, it is word-representable. \end{proof}

\begin{lemma}\label{lemma-K3-Kn}  The graph $A_3$ in Figure~\ref{thm-K3-Kn-pic} is a minimal non-word-representable graph. \end{lemma}

\begin{proof}
The graph $A_3$ is Graph 17 in Figure~3 in \cite{KS} containing minimal non-word-representable graphs on 7 vertices. In particular, the minimality follows from the fact that removing any vertex we do not obtain $W_5$, the only non-word-representable graph on 6 vertices. An alternative, short proof of the non-word-representability of $A_3$ goes as follows. Suppose, by Lemma~\ref{source-thm}, that vertex 10 is a source. Then, using Lemma~\ref{lemma} and the proof format from Section~\ref{format-subsec}, we obtain the following, where ``(10)'' refers to the vertex labeled 10: \\[-3mm]
	
\begin{footnotesize}
\noindent
{\bf 1.} B8$\rightarrow$9 (Copy 2) O8$\rightarrow$2 (C2(10)98) O3$\rightarrow$2 (C28(10)3) O3$\rightarrow$9 (C2893) O8$\rightarrow$4 (C2(10)48) O4$\rightarrow$9 (C3(10)49) O1$\rightarrow$9 O8$\rightarrow$1 (C1948) S:819(10)
	
\noindent
{\bf 2.} MC2 9$\rightarrow$8 O2$\rightarrow$8 (C2(10)98) O2$\rightarrow$3 (C28(10)3) O9$\rightarrow$3 (C2893) O4$\rightarrow$8 (C2(10)48) O9$\rightarrow$4 (C3(10)49) O9$\rightarrow$1 O1$\rightarrow$8 (C1948) S:918(10).
\end{footnotesize}
\end{proof}
	
\begin{theorem}\label{thm-K3-Kn}
	A $K_3$-$K_n$ graph is word-representable if and only if it does not contain $A_3$ in Figure~\ref{thm-K3-Kn-pic} as an induced subgraph.
\end{theorem}

\begin{proof} The graph $H_3$ on 10 vertices is shown in Figure~\ref{thm-K3-Kn-pic}, where the vertices marked by squares represent the clique $K_n$, with its edges omitted for a clearer visual representation of the graph. The graph $A_3$, also shown in Figure~\ref{thm-K3-Kn-pic}, is the only minimal forbidden subgraph of $H_3$. Indeed, using symmetries, we only need to consider four cases, where the first vertex removed corresponds to the smallest label removed.

\begin{itemize}
\item Vertex 1 is removed. By Theorem~\ref{same-neighbourhood}, we can remove vertices 5, 8, 9 as they have the same neighbourhoods as, respectively, vertices 4, 6, 7. The obtained graph, $A_1$, is on 6 vertices, and $A_1\neq W_5$, so $A_1$ is word-representable. In Figure~\ref{thm-K3-Kn-pic} we provide a semi-transitive orientation of $A_1$, hence giving an alternative justification for word-representability of $A_1$ by Theorem~\ref{fundamental}. Therefore, this case does not give us any minimal non-word-representable subgraphs.
\item Vertex 4 is removed. In Figure~\ref{thm-K3-Kn-pic}, we provide a semi-transitive orientation of the obtained graph, $A_2$. The fact that the orientation is semi-transitive can be checked directly via considering directed paths of length 3 or more beginning at vertices 1, 3, 5, 8 and 9 and applying the definition of a semi-transitive orientation. However, a simpler way to verify this claim is to use any of the available user-friendly software \cite{Glen,S}.
\item Vertex 5 is removed. Note that the obtained graph contains the graph $A_3$ in Figure~\ref{thm-K3-Kn-pic} as an induced subgraph, which is minimal non-word-representable by Lemma~\ref{lemma-K3-Kn}. To find other potentially minimal non-word-representable subgraphs, we need to ensure that $A_3$ is not a subgraph in the graph. This can be achieved by removing one of the vertices 8, 9, or 10. Since vertices 8 and 9 have the same neighborhoods, we only need to consider two cases:
\begin{itemize}
\item Vertex 8 is removed. The resulting graph is word-representable, as it is a subgraph of the word-representable graph $A_5$, which is considered below and presented in Figure~\ref{thm-K3-Kn-pic}.
\item Vertex 10 is removed. The obtained graph, $A_4$, is word-representable, and its semi-transitive orientation is presented in Figure~\ref{thm-K3-Kn-pic}. Checking that the orientation is semi-transitive can be done directly, or via software  \cite{Glen,S}.
\end{itemize}
\item Vertex 8 is removed. In Figure~\ref{thm-K3-Kn-pic}, we provide a semi-transitive orientation of the obtained graph, $A_5$. Again, semi-transitivity of the orientation can be checked directly, but it is more easily done using software  \cite{Glen,S}.
\end{itemize}

\vspace{-0.5cm}
\end{proof}
%-------------------------------------------%
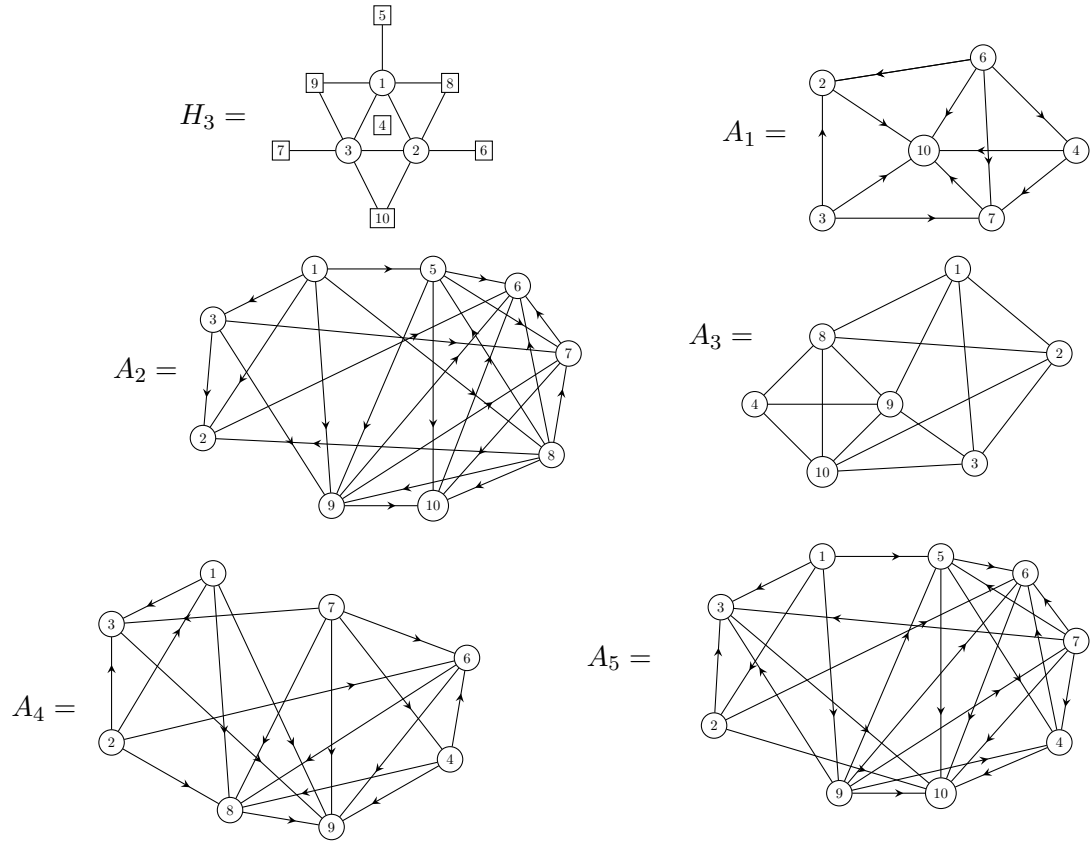
\begin{figure}
\begin{center}
\begin{tikzpicture}[scale=0.45]
			
\begin{scope}[shift={( -10,18)}]
	\draw (4,7) node [scale=0.5, circle,draw](node1){$1$};
	\draw (5,5) node [scale=0.5, circle,draw](node2){$2$};
	\draw (3,5) node [scale=0.5, circle, draw](node3){$3$};
\draw (4,5.75) node [scale=0.5, rectangle, draw](node4){$4$};
\draw (4,9) node [scale=0.5, rectangle,draw](node5){$5$};
\draw (7,5) node [scale=0.5, rectangle, draw](node6){$6$};
\draw (1,5) node [scale=0.5, rectangle, draw](node7){$7$};
\draw (6,7) node [scale=0.5, rectangle, draw](node8){$8$};
\draw (4,3) node [scale=0.5, rectangle, draw](node10){$10$};
\draw (2,7) node [scale=0.5, rectangle, draw](node9){$9$};
				\node at (-1,6) {$H_3=$};
				\draw (node1)--(node2)--(node3)--(node1);
				\draw (node1)--(node5);
				\draw (node2)--(node6);
				\draw (node3)--(node7);
				\draw (node1)--(node8);
				\draw (node2)--(node8);
				\draw (node1)--(node9);
				\draw (node3)--(node9);
				\draw (node2)--(node10);
				\draw (node3)--(node10);
\end{scope}
% A_1   removing vertex 1
	\begin{scope}[shift={(6,23)}]
	\draw (4,0) node [scale=0.5, circle, draw](node10){$10$};
	\draw (6,-2) node [scale=0.5, circle, draw](node7){$7$};
	\draw (1,2) node [scale=0.5, circle, draw](node2){$2$};
	\draw (8.5,0) node [scale=0.5, circle, draw](node4){$4$};
	\draw (5.75,2.75) node [scale=0.5, circle, draw](node6){$6$};
	\draw (1,-2) node [scale=0.5, circle, draw](node3){$3$};
	\node at (-1,0.5) {$A_1=$};
	\begin{scope}
	\draw [directed](node6)--(node2);
	\draw [directed](node6)--(node4);
	\draw [directed](node4)--(node7);
	\draw [directed](node3)--(node7);
	\draw [directed](node6)--(node2);
	\draw [directed](node2)--(node10);
	\draw [directed](node3)--(node10);
	\draw [directed](node3)--(node2);
	\draw [directed] (node6)--(node10);
	\draw [directed](node6)--(node7);
	\draw [directed](node4)--(node10);
	\draw [directed](node7)--(node10);
\end{scope}

\end{scope}
% A_2   removing vertex 4
\begin{scope}[shift={(-10,15.5)}]
	\draw (2,4) node [scale=0.5, circle, draw](node1){$1$};
	\draw (-1.3,-1) node [scale=0.5, circle, draw](node2){$2$};
	\draw (-1,2.5) node [scale=0.5, circle, draw](node3){$3$};
	
	\draw (5.5,4) node [scale=0.5, circle, draw](node5){$5$};
	\draw (8,3.5) node [scale=0.5, circle, draw](node6){$6$};
	\draw (9.5,1.5) node [scale=0.5, circle, draw](node7){$7$};
	\draw (9,-1.5) node [scale=0.5, circle, draw](node8){$8$};
	\draw (2.5,-3) node [scale=0.5, circle, draw](node9){$9$};
	\draw (5.5,-3) node [scale=0.5, circle, draw](node10){$10$};
	
	\node at (-3,1) {$A_2=$};
	\draw [directed](node1)--(node2);
	\draw [directed](node3)--(node2);
	\draw [directed](node1)--(node3);
	\draw [directed](node1)--(node5);
	\draw [directed](node2)--(node6);
	\draw [directed](node3)--(node7);
	\draw [directed](node1)--(node8);
	\draw [directed](node8)--(node2);
	\draw [directed](node1)--(node9);
	\draw [directed](node3)--(node9);
	\draw [directed](node5)--(node6);
	\draw [directed](node5)--(node7);
	\draw [directed](node8)--(node5);
	\draw [directed](node5)--(node9);
	\draw [directed](node5)--(node10);
	\draw [directed](node7)--(node6);
	\draw [directed](node8)--(node6);
	\draw [directed](node9)--(node6);
	\draw [directed](node10)--(node6);
	\draw [directed](node8)--(node7);
	\draw [directed](node9)--(node7);
	\draw [directed](node7)--(node10);
	\draw [directed](node8)--(node9);
	\draw [directed](node8)--(node10);
	\draw [directed](node9)--(node10);
	
\end{scope}
% Graph A_3 = 5.6.7
\begin{scope}[shift={(5,15.5)}]
	\draw (0,0) node [scale=0.5, circle, draw](node4){$4$};
	\draw (4,0) node [scale=0.5, circle, draw](node9){$9$};
	\draw (9,1.5) node [scale=0.5, circle, draw](node2){$2$};
	\draw (2,2) node [scale=0.5, circle, draw](node8){$8$};
	\draw (6,4) node [scale=0.5, circle, draw](node1){$1$};
	\draw (6.5,-1.75) node [scale=0.5, circle, draw](node3){$3$};
	\draw (2,-2) node [scale=0.5, circle, draw](node10){$10$};
	\node at (-1,2) {$A_3=$};
	\draw (node1)--(node2);
	\draw (node1)--(node3);
	\draw (node2)--(node3);
	\draw (node1)--(node9);
	\draw (node3)--(node9);
	\draw (node8)--(node9);
	\draw (node8)--(node4);
	\draw (node8)--(node10);
	\draw (node4)--(node9);
	\draw (node10)--(node4);
	\draw (node10)--(node9);
	\draw (node2)--(node10);
	\draw (node2)--(node8);
	\draw (node1)--(node8);
	\draw (node3)--(node10);

\end{scope}
% Graph A_4 = 5.10
\begin{scope}[shift={(-5,-9)}]
	\begin{scope}[shift={(-8,15.5)}]
		\draw (2,4) node [scale=0.5, circle, draw](node1){$1$};
		\draw (-1,-1) node [scale=0.5, circle, draw](node2){$2$};
		\draw (-1,2.5) node [scale=0.5, circle, draw](node3){$3$};
		
		\draw (5.5,3) node [scale=0.5, circle, draw](node7){$7$};
		\draw (9.5,1.5) node [scale=0.5, circle, draw](node6){$6$};
		\draw (9,-1.5) node [scale=0.5, circle, draw](node4){$4$};
		\draw (2.5,-3) node [scale=0.5, circle, draw](node8){$8$};
		\draw (5.5,-3.5) node [scale=0.5, circle, draw](node9){$9$};
		
		\node at (-3,0) {$A_4=$};
		\draw [directed](node2)--(node1);
		\draw [directed](node2)--(node3);
		\draw [directed](node1)--(node3);
		\draw [directed](node1)--(node8);
		\draw [directed](node2)--(node8);
		\draw [directed](node1)--(node9);
		\draw [directed](node3)--(node9);
		\draw [directed](node7)--(node3);
		\draw [directed](node2)--(node6);
		\draw [directed](node4)--(node6);
		\draw [directed](node7)--(node4);
		\draw [directed](node4)--(node8);
		\draw [directed](node4)--(node9);
		\draw [directed](node7)--(node6);
		\draw [directed](node6)--(node8);
		\draw [directed](node6)--(node9);
		\draw [directed](node7)--(node8);
		\draw [directed](node7)--(node9);
		\draw [directed](node8)--(node9);
	\end{scope}
	%\node at (-3,0) {$A_4=$};
\end{scope}
% A_5   removing vertex 8
\begin{scope}[shift={(5,7)}]
	\draw (2,4) node [scale=0.5, circle, draw](node1){$1$};
	\draw (-1.2,-1) node [scale=0.5, circle, draw](node2){$2$};
	\draw (-1,2.5) node [scale=0.5, circle, draw](node3){$3$};
	
	\draw (5.5,4) node [scale=0.5, circle, draw](node5){$5$};
	\draw (8,3.5) node [scale=0.5, circle, draw](node6){$6$};
	\draw (9.5,1.5) node [scale=0.5, circle, draw](node7){$7$};
	\draw (9,-1.5) node [scale=0.5, circle, draw](node4){$4$};
	\draw (2.5,-3) node [scale=0.5, circle, draw](node9){$9$};
	\draw (5.5,-3) node [scale=0.5, circle, draw](node10){$10$};
	
	\node at (-4,1) {$A_5=$};
	\draw [directed](node1)--(node2);
	\draw [directed](node2)--(node3);
	\draw [directed](node1)--(node3);
	\draw [directed](node1)--(node5);
	\draw [directed](node2)--(node6);
	\draw [directed](node7)--(node3);
	\draw [directed](node1)--(node9);
	\draw [directed](node9)--(node3);
	\draw [directed](node5)--(node6);
	\draw [directed](node7)--(node5);
	\draw [directed](node5)--(node4);
	\draw [directed](node9)--(node5);
	\draw [directed](node5)--(node10);
	\draw [directed](node7)--(node6);
	\draw [directed](node4)--(node6);
	\draw [directed](node9)--(node6);
	\draw [directed](node6)--(node10);
	\draw [directed](node7)--(node4);
	\draw [directed](node9)--(node7);
	\draw [directed](node7)--(node10);
	\draw [directed](node9)--(node4);
	\draw [directed](node4)--(node10);
	\draw [directed](node9)--(node10);
	\draw [directed](node2)--(node10);
	\draw [directed](node3)--(node10);
	
\end{scope}
														
\end{tikzpicture}
\end{center}
\caption{The graphs $H_3$ and $A_1$--$A_5$ in the proof of Theorem~\ref{thm-K3-Kn}. The vertices marked by squares in $H_3$ form a clique, with its edges omitted for a clearer visual representation of the graph.}\label{thm-K3-Kn-pic}
\end{figure}
		
\section{The word-representability of $K_4$-$K_n$ graphs}\label{sec-K4-Kn}

\begin{lemma}\label{lemma-K4-Kn} The subgraphs of $K_4$-$K_n$ in Figure~\ref{min-forb-sub-K4-Kn} are minimal non-word-representable. \end{lemma}

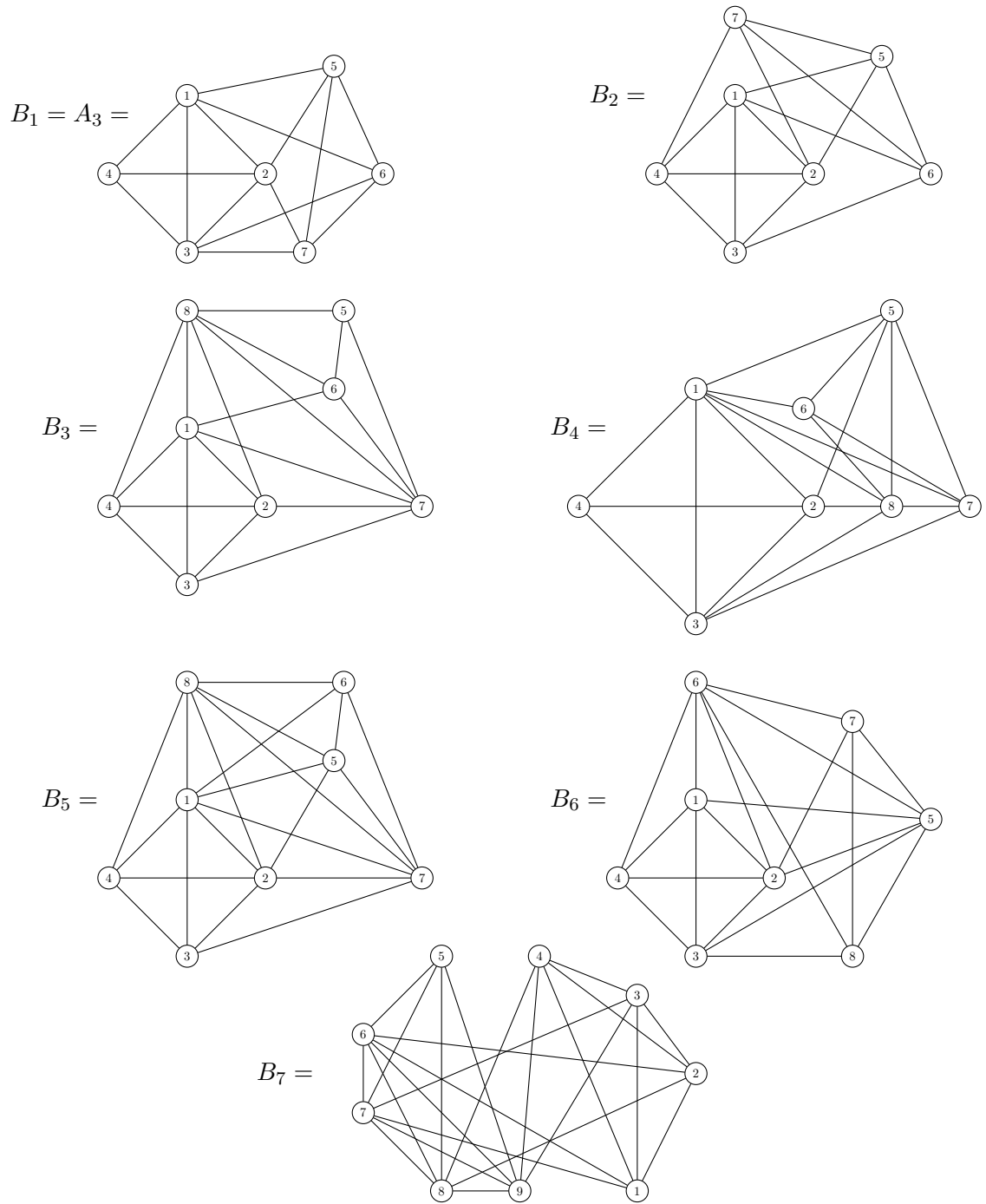
\begin{figure}
\begin{center}
\begin{tikzpicture}[scale=0.6]
	\begin{scope}[shift={(-8,16)}]
		\draw (0,0) node [scale=0.5, circle, draw](node4){$4$};
		\draw (4,0) node [scale=0.5, circle, draw](node2){$2$};
		\draw (5,-2) node [scale=0.5, circle, draw](node7){$7$};
		\draw (2,2) node [scale=0.5, circle, draw](node1){$1$};
		\draw (7,0) node [scale=0.5, circle, draw](node6){$6$};
		\draw (5.75,2.75) node [scale=0.5, circle, draw](node5){$5$};
		\draw (2,-2) node [scale=0.5, circle, draw](node3){$3$};
		\node at (-1,1.5) {$B_1=A_3=$};
\draw (node1)--(node2)--(node3)--(node4)--(node1)--(node3);
\draw (node2)--(node4);
\draw (node1)--(node5);
\draw (node2)--(node5);
\draw (node1)--(node6);
\draw (node3)--(node6);
\draw (node5)--(node7);
\draw (node6)--(node7);
\draw (node5)--(node6);
\draw (node2)--(node7);
\draw (node3)--(node7);
\end{scope}
\begin{scope}[shift={( 6,16)}]
		\draw (0,0) node [scale=0.5, circle, draw](node4){$4$};
		\draw (4,0) node [scale=0.5, circle, draw](node2){$2$};
		\draw (7,0) node [scale=0.5, circle, draw](node6){$6$};
		\draw (2,2) node [scale=0.5, circle, draw](node1){$1$};
		\draw (2,4) node [scale=0.5, circle, draw](node7){$7$};
		\draw (5.75,3) node [scale=0.5, circle, draw](node5){$5$};
		\draw (2,-2) node [scale=0.5, circle, draw](node3){$3$};
	\node at (-1,2) {$B_2=$};	
\draw (node7)--(node6);
\draw (node7)--(node5);
\draw (node6)--(node5);
\draw (node1)--(node2)--(node3)--(node4)--(node1)--(node3);
\draw (node2)--(node4);
\draw (node1)--(node5);
\draw (node2)--(node5);
\draw (node4)--(node7);
\draw (node2)--(node7);
\draw (node1)--(node6);
\draw (node3)--(node6);
\end{scope}
\begin{scope}[shift={( -8,7.5)}]
	\draw (0,0) node [scale=0.5, circle, draw](node4){$4$};
	\draw (4,0) node [scale=0.5, circle, draw](node2){$2$};
	\draw (8,0) node [scale=0.5, circle, draw](node7){$7$};
	\draw (2,2) node [scale=0.5, circle, draw](node1){$1$};
	\draw (2,5) node [scale=0.5, circle, draw](node8){$8$};
	\draw (6,5) node [scale=0.5, circle, draw](node5){$5$};
	\draw (5.75,3) node [scale=0.5, circle, draw](node6){$6$};
	\draw (2,-2) node [scale=0.5, circle, draw](node3){$3$};

	\node at (-1,2) {$B_3=$};
	\draw (node4)--(node1)--(node2)--(node3)--(node4)--(node2);
	\draw (node1)--(node3);
	\draw (node1)--(node8);
	\draw (node2)--(node8);
	\draw (node4)--(node8);
	\draw (node1)--(node7);
	\draw (node2)--(node7);
	\draw (node3)--(node7);
	\draw (node8)--(node5);
	\draw (node6)--(node5);
	\draw (node6)--(node7);
	\draw (node6)--(node8);
	\draw (node6)--(node1);
	\draw (node5)--(node7);
	\draw (node7)--(node8);
\end{scope}
\begin{scope}[shift={(5,7.5)}]
\draw (-1,0) node [scale=0.5, circle, draw](node4){$4$};
\draw (5,0) node [scale=0.5, circle, draw](node2){$2$};
\draw (9,0) node [scale=0.5, circle, draw](node7){$7$};
\draw (2,3) node [scale=0.5, circle, draw](node1){$1$};
\draw (4.75,2.5) node [scale=0.5, circle, draw](node6){$6$};
\draw (7,5) node [scale=0.5, circle, draw](node5){$5$};
\draw (7,0) node [scale=0.5, circle, draw](node8){$8$};
\draw (2,-3) node [scale=0.5, circle, draw](node3){$3$};

\node at (-1,2) {$B_4=$};
\draw (node4)--(node1)--(node2)--(node3)--(node4)--(node2);
\draw (node1)--(node3);
\draw (node1)--(node6);
\draw (node1)--(node5);
\draw (node1)--(node7);
\draw (node1)--(node8);
\draw (node2)--(node8);
\draw (node3)--(node8);
\draw (node3)--(node7);
\draw (node2)--(node5);
\draw (node5)--(node6);
\draw (node5)--(node8);
\draw (node5)--(node7);
\draw (node6)--(node8);
\draw (node6)--(node7);
\draw (node8)--(node7);
\end{scope}
\begin{scope}[shift={(-8,-2)}]
	\draw (0,0) node [scale=0.5, circle, draw](node4){$4$};
	\draw (4,0) node [scale=0.5, circle, draw](node2){$2$};
	\draw (8,0) node [scale=0.5, circle, draw](node7){$7$};
	\draw (2,2) node [scale=0.5, circle, draw](node1){$1$};
	\draw (2,5) node [scale=0.5, circle, draw](node8){$8$};
	\draw (6,5) node [scale=0.5, circle, draw](node6){$6$};
	\draw (5.75,3) node [scale=0.5, circle, draw](node5){$5$};
	\draw (2,-2) node [scale=0.5, circle, draw](node3){$3$};
	\node at (-1,2) {$B_5=$};
	\draw (node4)--(node1)--(node2)--(node3)--(node4)--(node2);
	\draw (node1)--(node3);
	\draw (node1)--(node8);
	\draw (node1)--(node6);
	\draw (node1)--(node5);
	\draw (node1)--(node7);
	\draw (node2)--(node7);
	\draw (node2)--(node5);
	\draw (node8)--(node7);
	\draw (node5)--(node7);
	\draw (node5)--(node6);
	\draw (node5)--(node8);
	\draw (node6)--(node8);
	\draw (node6)--(node7);
	\draw (node4)--(node8);
	\draw (node2)--(node8);
	\draw (node3)--(node7);
\end{scope}
\begin{scope}[shift={(5,-2)}]
	\draw (0,0) node [scale=0.5, circle, draw](node4){$4$};
	\draw (4,0) node [scale=0.5, circle, draw](node2){$2$};
	\draw (8,1.5) node [scale=0.5, circle, draw](node5){$5$};
	\draw (2,2) node [scale=0.5, circle, draw](node1){$1$};
	\draw (2,5) node [scale=0.5, circle, draw](node6){$6$};
	\draw (6,4) node [scale=0.5, circle, draw](node7){$7$};
	\draw (6,-2) node [scale=0.5, circle, draw](node8){$8$};
	\draw (2,-2) node [scale=0.5, circle, draw](node3){$3$};
	\node at (-1,2) {$B_6=$};
	\draw (node4)--(node1)--(node2)--(node3)--(node4)--(node2);
	\draw (node1)--(node3);
	\draw (node3)--(node5);
	\draw (node8)--(node5);
	\draw (node3)--(node8);
	\draw (node5)--(node1);
	\draw (node5)--(node2);
	\draw (node6)--(node1);
	\draw (node6)--(node2);
	\draw (node6)--(node4);
	\draw (node7)--(node6);
	\draw (node7)--(node5);
	\draw (node7)--(node8);
	\draw (node2)--(node7);
	\draw (node5)--(node6);
	\draw (node8)--(node6);
	
\end{scope}	
\begin{scope}[shift={(-1.5,-8)}]
	\draw (7,-2) node [scale=0.5, circle, draw](node1){$1$};
	\draw (8.5,1) node [scale=0.5, circle, draw](node2){$2$};
	\draw (7,3) node [scale=0.5, circle, draw](node3){$3$};
	\draw (4.5,4) node [scale=0.5, circle, draw](node4){$4$};
	\draw (2,4) node [scale=0.5, circle, draw](node5){$5$};
	\draw (0,2) node [scale=0.5, circle, draw](node6){$6$};
	\draw (0,0) node [scale=0.5, circle, draw](node7){$7$};
	\draw (2,-2) node [scale=0.5, circle, draw](node8){$8$};
	\draw (4,-2) node [scale=0.5, circle, draw](node9){$9$};
	
	\node at (-2,1) {$B_7=$};
	\draw (node1)--(node2)--(node3)--(node4)--(node1);
	\draw (node2)--(node4);
	\draw (node1)--(node3);
	\draw (node5)--(node6)--(node7)--(node8)--(node9)--(node5);
	\draw (node5)--(node7);
	\draw (node5)--(node8);
	\draw (node6)--(node8);
	\draw (node6)--(node9);
	\draw (node7)--(node9);
	\draw (node6)--(node1);
	\draw (node6)--(node2);
	\draw (node7)--(node1);
	\draw (node7)--(node3);
	\draw (node8)--(node2);
	\draw (node8)--(node4);
	\draw (node9)--(node3);
	\draw (node9)--(node4);

\end{scope}											
\end{tikzpicture}
\end{center}
\caption{Minimal forbidden subgraphs in $K_4$-$K_n$}\label{min-forb-sub-K4-Kn}
\end{figure}

\begin{proof}  The minimality of $B_1$ and $B_2$ follows from the fact that removing any vertex in it we have a graph on 6 vertices different from $W_5$, and hence word-representable. The minimality of $B_3$--$B_7$ follows from the fact that removing any vertex in it we have a graph on 7 vertices different from any of the 25 non-word-representable graphs in Figure~3 in \cite{KS}.  Alternatively, one can use  the user-friendly software \cite{Glen} to check that the obtained graphs by removing a single vertex are word-representable. 

Next, we prove non-word-representability of $B_1$--$B_7$. In each case we use Lemma~\ref{source-thm} to assume a particular vertex being a source, and the format of the proof presented in Section~\ref{format-subsec} (that relies on Lemma~\ref{lemma}). \\[-3mm]

\noindent
{\normalsize \bf Graph {$B_1$}}. Let vertex 1 be a source.\\[-3mm]
		
		\begin{footnotesize}
		\noindent {\bf 1.} B5$\rightarrow$6 (Copy 2) O5$\rightarrow$2 (C1652) O3$\rightarrow$2 (C1523) O3$\rightarrow$6 (C1632) O4$\rightarrow$2 (C1524) O3$\rightarrow$4 (C1634) O7$\rightarrow$2 O3$\rightarrow$7 (C2734) S:1372
		
		 \noindent {\bf 2.} MC2 6$\rightarrow$5 O2$\rightarrow$5 (C1652) O2$\rightarrow$3 (C1523) O6$\rightarrow$3 (C1632) O2$\rightarrow$4 (C1524) O4$\rightarrow$3 (C1634) O2$\rightarrow$7 O7$\rightarrow$3 (C2734) S:1273 \\[-3mm]
		\end{footnotesize}
		
\noindent {\normalsize \bf Graph \bfseries{$B_2$}}.  Let vertex 1 be a source. \\[-3mm]
		
		\begin{footnotesize}
		\noindent {\bf 1.} B6$\rightarrow$7 (Copy 2) O6$\rightarrow$3 (C1763) O2$\rightarrow$3 (C1632) O2$\rightarrow$5 (C1523) O6$\rightarrow$5 (C1652) O7$\rightarrow$5 (C1752) O2$\rightarrow$4 (C1524) O7$\rightarrow$4 (C1742) S:1674 
		
		\noindent {\bf 2.} MC2 7$\rightarrow$6 O3$\rightarrow$6 (C1763) O3$\rightarrow$2 (C1632) O5$\rightarrow$2 (C1523) O5$\rightarrow$6 (C1652) O5$\rightarrow$7 (C1752) O4$\rightarrow$2 (C1524) O4$\rightarrow$7 (C1742) S:1476 \\[-3mm]
		\end{footnotesize}
		
\noindent
{\normalsize \bf Graph \bfseries{$B_3$}}.  Let vertex 1 be a source. \\[-3mm]
		
		\begin{footnotesize}
		\noindent {\bf 1.} B6$\rightarrow$8 (Copy 2) O3$\rightarrow$8 (C1683) O3$\rightarrow$2 (C1832) O4$\rightarrow$8 (C1684) O4$\rightarrow$2 (C1842) O7$\rightarrow$2 (C1724) O7$\rightarrow$6 (C1672) S:1768
		 
		\noindent {\bf 2.} MC2 8$\rightarrow$6 O8$\rightarrow$3 (C1683) O2$\rightarrow$3 (C1832) O8$\rightarrow$4 (C1684) O2$\rightarrow$4 (C1842) O2$\rightarrow$7 (C1724) O6$\rightarrow$7 (C1672) S:1867 \\[-3mm]
		 \end{footnotesize}
		
\noindent {\normalsize \bf Graph \bfseries{$B_4$}}.  Let vertex 1 be a source.\\[-3mm]
		
		\begin{footnotesize}		
	\noindent {\bf 1.} B7$\rightarrow$8 (Copy 2) O2$\rightarrow$8 (C1782) O5$\rightarrow$8 (C1582) O3$\rightarrow$8 (C1583) O6$\rightarrow$8 (C1683) O2$\rightarrow$4 (C1824) O2$\rightarrow$6 (C1624) O5$\rightarrow$6 (C1562) O7$\rightarrow$6 (C1762) O2$\rightarrow$3 (C1623) O7$\rightarrow$3 (C1732) O7$\rightarrow$5 (C1573) O4$\rightarrow$3 (C1734) S:1438
	
	\noindent {\bf 2.} MC2 8$\rightarrow$7 O8$\rightarrow$2 (C1782) O8$\rightarrow$5 (C1582) O8$\rightarrow$3 (C1583) O8$\rightarrow$6 (C1683) O4$\rightarrow$2 (C1824) O6$\rightarrow$2 (C1624) O6$\rightarrow$5 (C1562) O6$\rightarrow$7 (C1762) O3$\rightarrow$2 (C1623) O3$\rightarrow$7 (C1732) O5$\rightarrow$7 (C1573) O3$\rightarrow$4 (C1734) S:1834 \\[-3mm]
	\end{footnotesize}

\noindent {\normalsize \bf Graph \bfseries{$B_5$}}.  Let vertex 1 be a source.\\[-3mm]
		
		\begin{footnotesize}
	\noindent 	{\bf 1.} B7$\rightarrow$8 (Copy 2) O7$\rightarrow$3 (C1873) O7$\rightarrow$5 (C1573) O7$\rightarrow$2 (C1572) O7$\rightarrow$6 (C1673) O4$\rightarrow$2 (C1724) O6$\rightarrow$2 (C1624) O6$\rightarrow$5 (C1562) O3$\rightarrow$2 (C1623) O8$\rightarrow$2 (C1823) O8$\rightarrow$5 (C1582) O4$\rightarrow$3 (C1734) O4$\rightarrow$8 (C1843) S:1485
		
	\noindent	{\bf 2.} MC2 8$\rightarrow$7 O3$\rightarrow$7 (C1873) O5$\rightarrow$7 (C1573) O2$\rightarrow$7 (C1572) O6$\rightarrow$7 (C1673) O2$\rightarrow$4 (C1724) O2$\rightarrow$6 (C1624) O5$\rightarrow$6 (C1562) O2$\rightarrow$3 (C1623) O2$\rightarrow$8 (C1823) O5$\rightarrow$8 (C1582) O3$\rightarrow$4 (C1734) O8$\rightarrow$4 (C1843) S:1584 \\[-3mm]
		\end{footnotesize}
		
\noindent {\normalsize \bf Graph \bfseries{$B_6$}}.  Let vertex 1 be a source. \\[-3mm]
		
		\begin{footnotesize}
	\noindent	{\bf 1.} B7$\rightarrow$8 (Copy 2) O7$\rightarrow$2 (C1872) O7$\rightarrow$5 (C1572) O7$\rightarrow$3 (C1573) O4$\rightarrow$2 (C1724) O4$\rightarrow$8 (C1842) O4$\rightarrow$3 (C1734) O5$\rightarrow$8 (C1584) O3$\rightarrow$8 (C1583) O3$\rightarrow$2 (C1832) O6$\rightarrow$2 O7$\rightarrow$6 (C2673) S:1762
		
	\noindent	{\bf 2.} MC2 8$\rightarrow$7 O2$\rightarrow$7 (C1872) O5$\rightarrow$7 (C1572) O3$\rightarrow$7 (C1573) O2$\rightarrow$4 (C1724) O8$\rightarrow$4 (C1842) O3$\rightarrow$4 (C1734) O8$\rightarrow$5 (C1584) O8$\rightarrow$3 (C1583) O2$\rightarrow$3 (C1832) O2$\rightarrow$6 O6$\rightarrow$7 (C2673) S:1267\\[-3mm]
		\end{footnotesize}
	
\noindent {\normalsize \bf Graph \bfseries{$B_7$}}.  Let vertex 7 be a source. \\[-3mm]
		
		\begin{footnotesize}
	\noindent	{\bf 1.} B7$\rightarrow$8 (Copy 2) O7$\rightarrow$3 (C3987) O4$\rightarrow$3 (C3794) O4$\rightarrow$8 (C3784) O7$\rightarrow$5 (C3957) O7$\rightarrow$6 (C3967) O5$\rightarrow$8 (C4958) O6$\rightarrow$8 (C4968) B5$\rightarrow$6 (Copy 3) O7$\rightarrow$1 O1$\rightarrow$6 (C1756) S:7168
		
	\noindent	{\bf 2.} MC3 6$\rightarrow$5 O2$\rightarrow$8 O6$\rightarrow$2 (C2856) S:7628
		
	\noindent	{\bf 3.} MC2 8$\rightarrow$7 O3$\rightarrow$7 (C3987) O3$\rightarrow$4 (C3794) O8$\rightarrow$4 (C3784) O5$\rightarrow$7 (C3957) O6$\rightarrow$7 (C3967) O8$\rightarrow$5 (C4958) O8$\rightarrow$6 (C4968) B5$\rightarrow$6 (Copy 4) O8$\rightarrow$2 O2$\rightarrow$6 (C2856) S:8267
		
	\noindent	{\bf 4.} MC4 6$\rightarrow$5 O1$\rightarrow$7 O6$\rightarrow$1 (C1756) S:8617
		\end{footnotesize}
			
\end{proof}

The proof of the following theorem follows the same approach as the proof of Theorem~\ref{thm-K3-Kn}. However, it involves a significantly larger number of cases. To keep the proof concise, we introduce a new proof format and move the two longest out of five cases to the Appendix.  

\begin{theorem}\label{K4-Kn-thm}
	A $K_4$-$K_n$ graph is word-representable if and only if it avoids the graphs $B_1$--$B_7$ in Figure~\ref{min-forb-sub-K4-Kn} as induced subgraphs.
\end{theorem}

\begin{figure}
\begin{center}
		\begin{tikzpicture}[scale=0.7]
		\begin{scope}[shift={( 5,-10)}]
		%              --------------    G --------------
		\draw (9,13) node [scale=0.5, circle,draw](node1){$1$};
		\draw (12,13) node [scale=0.5, circle,draw](node2){$2$};
		\draw (12,10) node [scale=0.5, circle, draw](node3){$3$};
		\draw (9,10) node [scale=0.5, circle, draw](node4){$4$};

		\draw (6,13) node [scale=0.5, rectangle,draw](node6){$6$};
		\draw (15,13) node [scale=0.5, rectangle, draw](node7){$7$};
		\draw (15,10) node [scale=0.5, rectangle, draw](node8){$8$};
		\draw (6,10) node [scale=0.5, rectangle, draw](node9){$9$};
		\draw (10.5,16) node [scale=0.5, rectangle, draw](node10){$10$};
		\draw (9.5,16.5) node [scale=0.5, rectangle, draw](node5){$5$};
		\draw (13,16.5) node [scale=0.5, rectangle, draw](node11){$11$};
		\draw (15,11.5) node [scale=0.5, rectangle, draw](node13){$13$};
		\draw (10.5,7) node [scale=0.5, rectangle, draw](node15){$15$};
		\draw (6,11.5) node [scale=0.5, rectangle, draw](node12){$12$};
		\draw (8.25,16) node [scale=0.5, rectangle, draw](node14){$14$};
		\draw (15,17) node [scale=0.5, rectangle, draw](node16){$16$};
		\draw (15,6) node [scale=0.5, rectangle, draw](node19){$19$};
		\draw (6,17) node [scale=0.5, rectangle, draw](node17){$17$};
		\draw (6,6) node [scale=0.5, rectangle, draw](node18){$18$};
		
		\draw (node1)--(node2);
		\draw (node1)--(node3);
		\draw (node1)--(node4);
		\draw (node2)--(node3);
		\draw (node2)--(node4);
		\draw (node4)--(node3);
		\draw (node1)--(node6);
		\draw (node2)--(node7);
		\draw (node3)--(node8);
		\draw (node4)--(node9);
		\draw (node1)--(node10);
		\draw (node2)--(node10);
		\draw (node1)--(node11);
		\draw (node3)--(node11);
		\draw (node1)--(node12);
		\draw (node4)--(node12);
		\draw (node2)--(node13);
		\draw (node3)--(node13);
		\draw (node2)--(node14);
		\draw (node4)--(node14);
		\draw (node3)--(node15);
		\draw (node4)--(node15);
		\draw (node1)--(node16);
		\draw (node2)--(node16);
		\draw (node3)--(node16);
		\draw (node1)--(node17);
		\draw (node2)--(node17);
		\draw (node4)--(node17);
		\draw (node1)--(node18);
		\draw (node3)--(node18);
		\draw (node4)--(node18);
		\draw (node2)--(node19);
		\draw (node3)--(node19);
		\draw (node4)--(node19);
		\node at (4,11.5) {$C=$};
	\end{scope}				
	\end{tikzpicture}
	\end{center}
	\caption{The most general case of a graph $K_4$-$K_n$}\label{most-general-case-K4-Kn}
\end{figure}
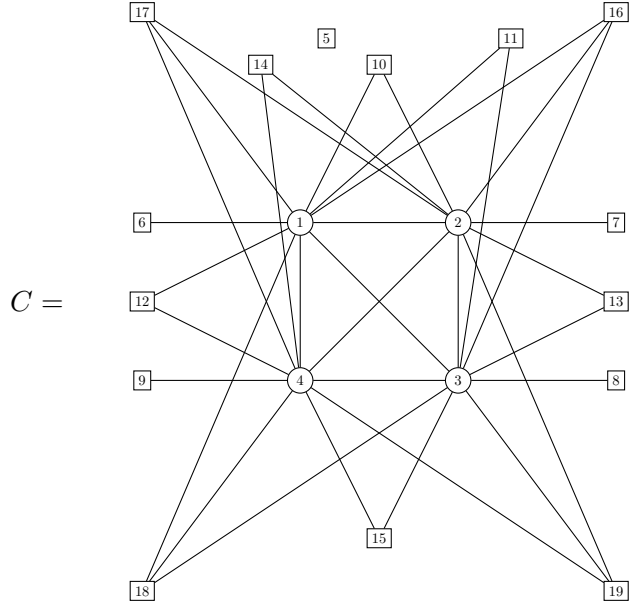

\begin{proof}
By Theorem~\ref{same-neighbourhood}, we can assume that the most general case of a graph $K_4$-$K_n$ is the graph $C$ in Figure~\ref{most-general-case-K4-Kn}, where again we use square boxes to indicate vertices connected to each other in the larger clique. 

Because of the symmetries, we need to consider four cases of removing a lexicographically smallest vertex to find all minimal non-word-representable subgraphs. \\[-3mm]

\noindent
{\bf Case 1:} Vertex $1$ is removed (which is equivalent to removing vertex  $2$, $3$, or $4$). After removing vertices having the same neighbourhoods and renaming the vertices, we obtain a graph isomorphic to the graph $H_3$ in Figure~\ref{thm-K3-Kn-pic}, which is the most general $K_3$-$K_n$ graph considered in Theorem~\ref{thm-K3-Kn}. Therefore, in this case, we obtain a single minimal non-word-representable subgraph, which is $A_3=B_1$.  \\[-3mm]

\noindent
{\bf Case 2:} Vertex $5$ is removed;  vertices $1$--$4$ are retained.  In this case, we obtain the minimal non-word-representable graphs $B_1$--$B_6$. See the Appendix for details. \\[-3mm]

\noindent
{\bf Case 3:} Vertex $6$ is removed (which is equivalent to removing vertex  $7$, $8$, or $9$);  vertices $1$--$5$ are retained. In this case, we obtain the minimal non-word-representable graphs $B_1$--$B_7$. See the Appendix for details. \\[-3mm]

\noindent
{\bf Case 4:}  Vertex $10$ is removed (which is equivalent to removing a vertex  in $\{11,12, 13, 14, 15\}$); vertices $1$--$9$ are retained.  Note that since there is an occurrence of $B_7$ formed by the vertices 1, 2, 3, 4, 5, 11, 12, 13, and 14, we do not need to consider the cases involving the removal of vertices labeled higher than 14 (and retaining the vertices 1--9 and 11--14). Moreover, by symmetry, removing vertex 11 is equivalent to removing a vertex in $\{12, 13, 14\}$, so without loss of generality, we assume that vertices 10 and 11 are removed. Using the fact that removing vertices 13 and 14 gives a graph isomorphic to one obtained by removing the vertices 13 and 15, there are six subcases to consider, and in all of them, we obtain graphs containing $B_2$, formed by 2, 3, 4, 7, 8, 9, and 19: 10--17, 10--14.16--18,  10--13.16--18, 10.11.13.14.16--18, 10.11.14--18, and 10.11.14.16--18; see the Appendix for an explanation of the shorthand notation used here. Hence in Case~4, no other minimal non-word-representable graphs exist other than those in Figure~\ref{min-forb-sub-K4-Kn}. \\[-3mm]

\noindent
{\bf Case 5:}  Vertex $16$ is removed (which is equivalent to removing vertex $17$, $18$, or $19$); vertices $1$--$15$ are retained. This case does not give any new minimal non-word-representable subgraphs, as all graphs in this case contain, for example, a copy of the non-word-representable subgraph $B_1$, formed by vertices 1, 2, 3, 4, 10, 11, and 13 as an induced subgraph. \\[-3mm]

Hence the graphs in $B_1$-$B_7$ in Figure~\ref{min-forb-sub-K4-Kn}  are the only minimal forbidden non-word-representable subgraphs in $K_4$-$K_n$.
\end{proof}

\section{Concluding remarks}\label{concluding}

In this paper, we extended the study of word-representability of split graphs to $K_m$-$K_n$ graphs by characterizing word-representable $K_m$-$K_n$ graphs,  in terms of forbidden subgraphs, for any $1\leq m\leq 4$ and $n\geq 1$. Extending our studies to the cases of $m\geq 5$ presents a challenging problem with our current approach, and we leave this as an open direction for further research.   

\section{Acknowledgments} The first author's research was supported by the China Scholarship Council (CSC) (No. 202308500094) and the Science and Technology Research Program of the Chongqing Municipal Education Commission (No. KJQN202100508).

% Appendix start here
\appendix
\renewcommand{\thesection}{Appendix.}

\section{Cases 2 \& 3 in the proof of Theorem~\ref{K4-Kn-thm}}

In what follows, we display all possible scenarios for removing one vertex at a time, using symmetries whenever possible to reduce the number of cases to consider. The notation $x_1.x_2.\ldots.x_k$, where $x_1<x_2<\cdots<x_k$, is used to indicate the non-word-representable graph obtained by first removing vertex $x_1$, then vertex $x_2$, and so on, and finally vertex $x_k$, with the property that removing any other vertex with label larger than $x_k$ results in a word-representable graph. This can be verified using software~\cite{Glen, S}.  In each case, we  give a reason for the graph being non-word-representable by indicating which of the graphs in $\{B_1,\ldots,B_7\}$ is contained in it as an induced subgraph (we list vertices forming such a subgraph). Often, we use shorthand notation by replacing $x.(x+1).\cdots.y$ with $x$--$y$; for example, 6.10--14.16.17 represents 6.10.11.12.13.14.16.17. 

Note also that, for example, removing 6--14 and 6--13.15 results in isomorphic graphs; therefore, it is sufficient to consider only the removal of 6--14 (the lexicographically smallest case). We frequently and implicitly use similar observations throughout the proof.   \\[-3mm]

\noindent
{\bf Case 2:}   \\[-3mm]

\begin{footnotesize}
	
\noindent
 $\bullet$ 5--10; contains $B_1$ formed by 1, 2, 3, 4, 11, 12, 15;\\[-3mm]
 
\noindent
 $\bullet$ 5--11; contains $B_1$ formed by 1, 2, 3, 4, 13, 14, 15;\\[-3mm]
 
\noindent
 $\bullet$ 5--12; contains $B_1$ formed by 1, 2, 3, 4, 13, 14, 15;\\[-3mm]
 
\noindent
 $\bullet$ 5--11.13; contains $B_2$ formed by 1, 2, 3, 4, 12, 14, 15;\\[-3mm]
 
\noindent
$\bullet$ 5--8.10--13; contains $B_1$ formed by 1, 2, 3, 9, 17, 18, 19;\\[-3mm]
 
\noindent
 $\bullet$ 5--8.10--14; contains $B_1$ formed by 1, 2, 3, 9, 17, 18, 19;\\[-3mm]

\noindent
 $\bullet$ 5--8.10--15; contains $B_1$ formed by 1, 2, 3, 9, 17, 18, 19;\\[-3mm]

\noindent
 $\bullet$ 5--8.10--14.16; contains $B_1$ formed by 1, 2, 3, 9, 17, 18, 19;\\[-3mm]
 
\noindent
 $\bullet$ 5--8.10--14.16.17; contains $B_5$ formed by 1, 2, 3, 4, 9, 15, 18, 19;\\[-3mm]
 
\noindent
 $\bullet$ 5--8.10--14.17; contains $B_5$ formed by 1, 2, 3, 4, 9, 15, 18, 19;\\[-3mm]
 
\noindent
 $\bullet$ 5--8.10--13.16; contains $B_1$ formed by 1, 2, 3, 9, 17, 18, 19;\\[-3mm]
 
\noindent
 $\bullet$ 5--8.10--13.16--18; contains $B_4$ formed by 1, 2, 3, 4, 9, 14, 15, 19;\\[-3mm]
 
\noindent
 $\bullet$ 5--8.10--13.16.17; contains $B_4$ formed by 1, 2, 3, 4, 9, 14, 15, 19;\\[-3mm]
 
\noindent
 $\bullet$ 5--8.10--13.17; contains $B_4$ formed by 1, 2, 3, 4, 9, 14, 15, 19;\\[-3mm]
 
\noindent
 $\bullet$ 5--8.10--13.17.18; contains $B_4$ formed by 1, 2, 3, 4, 9, 14, 15, 19;\\[-3mm]
 
\noindent
 $\bullet$ 5--8.10--12.14; contains $B_1$ formed by 1, 2, 3, 9, 13, 17, 18;\\[-3mm]
 
\noindent
 $\bullet$ 5--8.10--12.14.15; contains $B_1$ formed by 1, 2, 3, 9, 13, 17, 18;\\[-3mm]
 
\noindent
 $\bullet$ 5--8.10--12.14.16; contains $B_1$ formed by 1, 2, 3, 9, 13, 17, 18;\\[-3mm]
 
\noindent
 $\bullet$ 5--8.10--12.14.16.17; contains $B_1$ formed by 1, 2, 3, 4, 9, 15, 18, 19;\\[-3mm]
 
\noindent
 $\bullet$ 5--8.10.11.13.14; contains $B_1$ formed by 1, 2, 3, 9, 17, 18, 19;\\[-3mm]
 
\noindent
 $\bullet$ 5--8.10.11.13--15; contains $B_1$ formed by 1, 2, 3, 9, 17, 18, 19;\\[-3mm]
 
\noindent
 $\bullet$ 5--8.10.11.13--16; contains $B_1$ formed by 1, 2, 3, 9, 17, 18, 19;\\[-3mm]
 
\noindent
 $\bullet$ 5--8.10.11.13.14.16; contains $B_1$ formed by 1, 2, 3, 9, 17, 18, 19;\\[-3mm]
 
\noindent
 $\bullet$ 5--8.10.11.13.14.16.17; contains $B_4$ formed by 1, 2, 3, 4, 9, 12, 15, 18;\\[-3mm]
 
\noindent
 $\bullet$ 5--8.10.11.13.14.17; contains $B_4$ formed by 1, 2, 3, 4, 9,  12, 15, 18;\\[-3mm]
 
\noindent
 $\bullet$ 5--8.10.11.14; contains $B_1$ formed by 1, 2, 3, 9, 13, 17, 18;\\[-3mm]
 
\noindent
 $\bullet$ 5--8.10.11.14.15; contains $B_1$ formed by 1, 2, 3, 9, 13, 17, 18;\\[-3mm]
 
\noindent
 $\bullet$ 5--8.10.11.14--16; contains $B_1$ formed by 1, 2, 3, 9, 13, 17, 18;\\[-3mm]
 
\noindent
 $\bullet$ 5--8.10.11.14.16; contains $B_1$ formed by 1, 2, 3, 9, 13, 17, 18;\\[-3mm]
 
\noindent
 $\bullet$ 5--8.10.11.14.16.17; contains $B_4$ formed by 1, 2, 3, 4, 9, 12, 15, 18;\\[-3mm]
 
\noindent
 $\bullet$ 5--8.10.11.14.17; contains $B_4$ formed by 1, 2, 3, 4, 9, 12, 15, 18;\\[-3mm]
 
\noindent
 $\bullet$ 5--8.10.12--14; contains $B_1$ formed by 1, 2, 3, 9, 11, 17, 19;\\[-3mm]
 
\noindent
 $\bullet$ 5--8.10.12--15; contains $B_1$ formed by 1, 2, 3, 9, 11, 17, 19;\\[-3mm]
 
\noindent
 $\bullet$ 5--8.10--14.16; contains $B_1$ formed by 1, 2, 3, 9, 11, 17, 19;\\[-3mm]
 
\noindent
 $\bullet$ 5--8.10.12--14.16.17; contains $B_5$ formed by 1, 2, 3, 4, 9, 15, 18, 19;\\[-3mm]
 
\noindent
 $\bullet$ 5--8.10.12--14.16.17.18; contains $B_6$ formed by 1, 2, 3, 4, 9, 11, 15, 19;\\[-3mm]
 
\noindent
 $\bullet$ 5--8.10.12--14.16.18; contains $B_1$ formed by 1, 2, 3, 9, 11, 17, 19;\\[-3mm]
 
\noindent
 $\bullet$ 5--8.10.12--14.18; contains $B_1$ formed by 1, 2, 3, 9, 11, 17, 19;\\[-3mm]
 
\noindent
 $\bullet$ 5--8.10.12.13.16; contains $B_1$ formed by 1, 2, 3, 9, 11, 17, 19;\\[-3mm]
 
\noindent
 $\bullet$ 5--8.10.12.13.15.16; contains $B_1$ formed by 1, 2, 3, 9, 11, 17, 19;\\[-3mm]
 
\noindent
 $\bullet$ 5--8.10.12.13.16--18; contains $B_4$ formed by 1, 2, 3, 4, 9, 14, 15, 19;\\[-3mm]
 
\noindent
 $\bullet$ 5--8.10.12.13.16.18; contains $B_1$ formed by 1, 2, 3, 9, 11, 17, 19;\\[-3mm]
 
\noindent
 $\bullet$ 5--8.10.12.13.17.18; contains $B_4$ formed by 1, 2, 3, 4, 9, 14, 15, 19;\\[-3mm]
 
\noindent
 $\bullet$ 5--8.10.12.13.18; contains $B_1$ formed by 1, 2, 3, 9, 11, 17, 19;\\[-3mm]
 
\noindent
 $\bullet$ 5--8.10.12.15.16; contains $B_1$ formed by 1, 2, 3, 9, 11, 13, 17;\\[-3mm]
 
\noindent
 $\bullet$ 5--8.10.15; contains $B_1$ formed by 1, 2, 3, 9, 11, 13, 17;\\[-3mm]
 
\noindent
 $\bullet$ 5--8.10.15.16; contains $B_1$ formed by 1, 2, 3, 9, 11, 13, 17;\\[-3mm]
 
\noindent
 $\bullet$ 5--8.12--14.15; contains $B_1$ formed by 1, 2, 3, 9, 10, 11, 19;\\[-3mm]
 
\noindent
 $\bullet$ 5--8.12--17; contains $B_1$ formed by 1, 2, 3, 9, 10, 11, 19;\\[-3mm]
 
\noindent
 $\bullet$ 5--8.12--18; contains $B_1$ formed by 1, 2, 3, 9, 10, 11, 19;\\[-3mm]
 
\noindent
 $\bullet$ 5--8.12--14.16.17; contains $B_1$ formed by 1, 2, 3, 9, 10, 11, 19;\\[-3mm]
 
\noindent
 $\bullet$ 5--8.12--14.16--18; contains $B_1$ formed by 1, 2, 3, 9, 10, 11, 19;\\[-3mm]
 
\noindent
 $\bullet$ 5--8.12--14.16.18; contains $B_1$ formed by 1, 2, 3, 9, 10, 11, 19;\\[-3mm]
 
\noindent
 $\bullet$ 5--8.12--14.18; contains $B_1$ formed by 1, 2, 3, 9, 10, 11, 19;\\[-3mm]
 
\noindent
 $\bullet$ 5--8.12.13.16.17; contains $B_1$ formed by 1, 2, 3, 9, 10, 11, 19;\\[-3mm]
 
\noindent
 $\bullet$ 5--8.12.13.17.18; contains $B_1$ formed by 1, 2, 3, 9, 10, 11, 19;\\[-3mm]
 
\noindent
 $\bullet$ 5--7.10--16; contains $B_1$ formed by 1, 2, 3, 9, 17, 18, 19;\\[-3mm]
 
\noindent
 $\bullet$ 5--7.10--15.18; contains $B_6$ formed by 1, 2, 3, 4, 8, 9, 16, 19;\\[-3mm]
 
\noindent
 $\bullet$ 5--7.10--14.16.17.18; contains $B_3$ formed by 1, 2, 3, 4, 8, 9, 15, 19;\\[-3mm]
 
\noindent
 $\bullet$ 5--7.10--14.16.18; contains $B_3$ formed by 1, 2, 3, 4, 8, 9, 15, 19;\\[-3mm]
 
\noindent
 $\bullet$ 5--7.10--14.18; contains $B_3$ formed by 1, 2, 3, 4, 8, 9, 15, 19;\\[-3mm]
 
\noindent
 $\bullet$ 5--7.10--13.15.16; contains $B_1$ formed by 1, 2, 3, 9, 17, 18, 19;\\[-3mm]
 
\noindent
 $\bullet$ 5--7.10--13.15.17.18; contains $B_4$ formed by 1, 2, 3, 8, 9, 14, 16, 19;\\[-3mm]
 
\noindent
 $\bullet$ 5--7.10--13.16--18; contains $B_3$ formed by 1, 2, 3, 4, 8, 9, 15, 19;\\[-3mm]
 
\noindent
 $\bullet$ 5--7.10--13.16.18; contains $B_3$ formed by 1, 2, 3, 4, 8, 9, 15, 19;\\[-3mm]
 
\noindent
 $\bullet$ 5--7.10--13.17.18; contains $B_3$ formed by 1, 2, 3, 4, 8, 9, 15, 19;\\[-3mm]
 
\noindent
 $\bullet$ 5--7.10--13.18; contains $B_3$ formed by 1, 2, 3, 4, 8, 9, 15, 19;\\[-3mm]
 
\noindent
 $\bullet$ 5--7.10.15.16; contains $B_1$ formed by 1, 2, 3, 9, 11, 13, 17;\\[-3mm]
 
\noindent
 $\bullet$ 5--7.11--18; contains $B_2$ formed by 2, 3, 4, 8, 9, 10, 19;\\[-3mm]
 
\noindent
 $\bullet$ 5--7.11--14.16--18; contains $B_2$ formed by 2, 3, 4, 8, 9, 10, 19;\\[-3mm]
 
\noindent
 $\bullet$ 5--7.11--14.16.18; contains $B_2$ formed by 2, 3, 4, 8, 9, 10, 19;\\[-3mm]
 
\noindent
 $\bullet$ 5--7.11--14.18; contains $B_2$ formed by 2, 3, 4, 8, 9, 10, 19;\\[-3mm]
 
\noindent
 $\bullet$ 5--7.11.13--18; contains $B_1$ formed by 1, 2, 4, 8, 10, 12, 19;\\[-3mm]
 
\noindent
 $\bullet$ 5--7.11.13--15.17.18; contains $B_1$ formed by 1, 2, 4, 8, 10, 12, 19;\\[-3mm]
 
\noindent
 $\bullet$ 5--7.11.13.14.16--18; contains $B_1$ formed by 1, 2, 4, 8, 10, 12, 19;\\[-3mm]
 
\noindent
 $\bullet$ 5--7.11.13.14.16.18; contains $B_1$ formed by 1, 2, 4, 8, 10, 12, 19;\\[-3mm]
 
\noindent
 $\bullet$ 5--7.11.13.14.17.18; contains $B_1$ formed by 1, 2, 4, 8, 10, 12, 19;\\[-3mm]
 
\noindent
 $\bullet$ 5--7.11.13.14.18; contains $B_1$ formed by 1, 2, 4, 8, 10, 12, 19;\\[-3mm]
 
\noindent
 $\bullet$ 5--7.11.14--18; contains $B_1$ formed by 1, 2, 4, 8, 10, 12, 19;\\[-3mm]
 
\noindent
 $\bullet$ 5--7.11.14.16--18; contains $B_1$ formed by 1, 2, 4, 8, 10, 12, 19;\\[-3mm]
 
\noindent
 $\bullet$ 5--7.11.14.16.18; contains $B_1$ formed by 1, 2, 4, 8, 10, 12, 19;\\[-3mm]
 
\noindent
 $\bullet$ 5--7.11.14.18; contains $B_1$ formed by 1, 2, 4, 8, 10, 12, 19;\\[-3mm]
 
\noindent
 $\bullet$ 5--7.10.11.13.14--16; contains $B_1$ formed by 1, 2, 3, 9, 17, 18, 19;\\[-3mm]
 
\noindent
 $\bullet$ 5--7.10.11.13--15.17.18; contains $B_1$ formed by 1, 2, 4, 8, 12, 16, 19;\\[-3mm]
 
\noindent
 $\bullet$ 5--7.10.11.13--15.18; contains $B_1$ formed by 1, 2, 4, 8, 12, 16, 19;\\[-3mm]
 
\noindent
 $\bullet$ 5--7.10.11.13.14.16--18; contains $B_3$ formed by 1, 2, 3, 4, 8, 9, 15, 19;\\[-3mm]
 
\noindent
 $\bullet$ 5--7.10.11.13.14.16.18; contains $B_3$ formed by 1, 2, 3, 4, 8, 9, 15, 19;\\[-3mm]
 
\noindent
 $\bullet$ 5--7.10.11.13.14.17.18; contains $B_1$ formed by 1, 2, 4, 8, 12, 16, 19;\\[-3mm]
 
\noindent
 $\bullet$ 5--7.10.11.13.14.18; contains $B_1$ formed by 1, 2, 4, 8, 12, 16, 19;\\[-3mm]
 
\noindent
 $\bullet$ 5--7.10.11.13.15.16; contains $B_1$ formed by 1, 2, 3, 9, 17, 18, 19;\\[-3mm]
 
\noindent
 $\bullet$ 5--7.10.11.14.15.16; contains $B_1$ formed by 1, 2, 3, 9, 13, 17, 18;\\[-3mm]
 
\noindent
 $\bullet$ 5--7.10.11.14.16--18; contains $B_3$ formed by 1, 2, 3, 4, 8, 9, 15, 19;\\[-3mm]
 
\noindent
 $\bullet$ 5--7.10.11.14.16.18; contains $B_3$ formed by 1, 2, 3, 4, 8, 9, 15, 19;\\[-3mm]
 
\noindent
 $\bullet$ 5--7.10.11.14.15.18; contains $B_1$ formed by 1, 2, 4, 8, 12, 16, 19;\\[-3mm]
 
\noindent
 $\bullet$ 5--7.10.11.14.18; contains $B_1$ formed by 1, 2, 4, 8, 12, 16, 19;\\[-3mm]
 
\noindent
 $\bullet$ 5--7.10.11.15.16; contains $B_1$ formed by 1, 2, 3, 9, 13, 17, 18;\\[-3mm]
 
\noindent
 $\bullet$ 5.6.10--17; contains $B_2$ formed by 2, 3, 4, 7, 8, 9, 19;\\[-3mm]
 
\noindent
 $\bullet$ 5.6.10--18; contains $B_2$ formed by 2, 3, 4, 7, 8, 9, 19;\\[-3mm]
 
\noindent 
 $\bullet$ 5.6.10--14.16--18; contains $B_2$ formed by 2, 3, 4, 7, 8, 9, 19;\\[-3mm]
 
\noindent
 $\bullet$ 5.6.10--14.16.18; contains $B_2$ formed by 2, 3, 4, 7, 8, 9, 19;\\[-3mm]
 
\noindent
 $\bullet$ 5.6.10--13.16--18; contains $B_2$ formed by 2, 3, 4, 7, 8, 9, 19;\\[-3mm]
 
\noindent
 $\bullet$ 5.6.10--13.17.18; contains $B_2$ formed by 2, 3, 4, 7, 8, 9, 19;\\[-3mm]
 
\noindent
 $\bullet$ 5.6.10.11.13--18; contains $B_2$ formed by 2, 3, 4, 7, 8, 9, 19;\\[-3mm]
 
\noindent
 $\bullet$ 5.6.10.11.13--16.18; contains $B_2$ formed by 2, 3, 4, 7, 8, 9, 19;\\[-3mm]
 
\noindent
 $\bullet$ 5.6.10.11.13.14.16--18; contains $B_2$ formed by 2, 3, 4, 7, 8, 9, 19;\\[-3mm]
 
\noindent
 $\bullet$ 5.6.10.11.13.14.16.18; contains $B_2$ formed by 2, 3, 4, 7, 8, 9, 19;\\[-3mm]
 
\noindent
 $\bullet$ 5.6.10.11.14.15--17; contains $B_2$ formed by 2, 3, 4, 7, 8, 9, 19;\\[-3mm]
 
\noindent
 $\bullet$ 5.6.10.11.14.15--18; contains $B_2$ formed by 2, 3, 4, 7, 8, 9, 19;\\[-3mm]
 
\noindent
 $\bullet$ 5.6.10.11.14.16--18; contains $B_2$ formed by 2, 3, 4, 7, 8, 9, 19;\\[-3mm]
 
\noindent
 $\bullet$ 5.6.10.11.14.16.18; contains $B_2$ formed by 2, 3, 4, 7, 8, 9, 19;\\[-3mm]
 
\noindent
 $\bullet$ 5.6.10.13--18; contains $B_1$ formed by 1, 3, 4, 7, 11, 12, 19;\\[-3mm]
 
\noindent
 $\bullet$ 5.6.10.13--17; contains $B_1$ formed by 1, 3, 4, 7, 11, 12, 19;\\[-3mm]
 
\noindent
 $\bullet$ 5.6.10.13.15--18; contains $B_1$ formed by 1, 3, 4, 7, 11, 12, 19;\\[-3mm]
 
\noindent
 $\bullet$ 5.6.10.15--18; contains $B_1$ formed by 1, 3, 4, 7, 11, 12, 19;\\[-3mm]
 
\noindent
 $\bullet$ 5.10--17.18; contains $B_1$ formed by 1, 2, 3, 9, 17, 18, 19;\\[-3mm]
 
\noindent
 $\bullet$ 5.10--13.16--18; contains $B_4$ formed by 1, 2, 3, 4, 9, 14, 15, 19;\\[-3mm]
 
\noindent
 $\bullet$ 5.10--14.16--18; contains $B_2$ formed by 2, 3, 4, 7, 8, 9, 19;\\[-3mm]
 
\noindent
 $\bullet$ 5.10.11.13--18; contains $B_2$ formed by 2, 3, 4, 7, 8, 9, 19;\\[-3mm]
 
\noindent
 $\bullet$ 5.10.11.13.14.16--18; contains $B_2$ formed by 2, 3, 4, 7, 8, 9, 19;\\[-3mm]
 
\noindent
 $\bullet$ 5.10.11.13.16--18; contains $B_2$ formed by 2, 3, 4, 7, 8, 9, 19;\\[-3mm]
 
\noindent
 $\bullet$ 5.10.11.14.15--18; contains $B_2$ formed by 2, 3, 4, 7, 8, 9, 19;\\[-3mm]
 
\noindent
 $\bullet$ 5.10.11.14.16--18; contains $B_2$ formed by 2, 3, 4, 7, 8, 9, 19;\\[-3mm]
 
\noindent
 $\bullet$ 5.10.13--18; contains $B_1$ formed by 1, 3, 4, 7, 11, 12, 19;\\[-3mm]
 
 \noindent
 $\bullet$ 5.10.13.15--18; contains $B_1$ formed by 1, 3, 4, 7, 11, 12, 19;\\[-3mm]
 
 \noindent
 $\bullet$ 5.10.15--18; contains $B_1$ formed by 2, 3, 4, 7, 8, 9, 19.\\[-2mm]
 
 \end{footnotesize}

\noindent
{\bf Case 3:} \\[-3mm] 

\begin{footnotesize}

\noindent
$\bullet$ 6--16; contains $B_1$ formed by 1, 2, 3, 5, 17, 18, 19; \\[-3mm]

\noindent
$\bullet$ 6--14.16; contains $B_1$ formed by 1, 2, 3, 5, 17, 18, 19; \\[-3mm]

\noindent
$\bullet$ 6--11.13--16; contains $B_1$ formed by 1, 2, 3, 5, 17, 18, 19; \\[-3mm]

\noindent
 $\bullet$ 6--11.13--15.17; contains $B_1$ formed by 1, 2, 4, 5, 12, 16, 19;\\[-3mm]
 
\noindent
 $\bullet$ 6--11.13--15.17.18; contains $B_1$ formed by 1, 2, 4, 5, 12, 16, 19;\\[-3mm]

\noindent
 $\bullet$ 6--11.13.14.16; contains $B_1$ formed by 1, 2, 3, 5, 17, 18, 19;\\[-3mm]
 
\noindent
 $\bullet$ 6--11.14--16; contains $B_1$ formed by 1, 2, 3, 5, 13, 17, 18;\\[-3mm]

\noindent
 $\bullet$ 6--11.14.16; contains $B_1$ formed by 2, 3, 4, 5, 13, 15, 17;\\[-3mm]

\noindent
 $\bullet$ 6--8.10--17; contains $B_3$ formed by 1, 2, 3, 4, 5, 9, 18, 19;\\[-3mm]

\noindent
 $\bullet$ 6--8.10--16.18; contains $B_3$ 1, 2, 3, 4, 5, 9, 17, 19;\\[-3mm]

\noindent
 $\bullet$ 6--8.10--15.17; contains $B_1$ formed by 1, 2, 4, 5, 16, 18, 19;\\[-3mm]

\noindent
 $\bullet$ 6--8.10--14.16.17; contains $B_3$ formed by 1, 2, 3, 4, 5, 9, 18, 19;\\[-3mm]

\noindent
 $\bullet$ 6--8.10--14.16.18; contains $B_3$ formed by 1, 2, 3, 4, 5, 9, 17, 19;\\[-3mm]

\noindent
 $\bullet$ 6--8.10--14.17; contains $B_1$ formed by 1, 2, 4, 5, 16, 18, 19;\\[-3mm]

\noindent
 $\bullet$ 6--8.10--13.16--18; contains $B_4$ formed by 1, 2, 3, 4, 9, 14, 15, 19;\\[-3mm]

\noindent
 $\bullet$ 6--8.10--13.16.17.19; contains $B_3$ formed by 1, 2, 3, 4, 5, 9, 18, 19; \\[-3mm]

\noindent
 $\bullet$ 6--8.10--13.16.18; contains $B_3$ formed by 1, 2, 3, 4, 5, 9, 17, 19;\\[-3mm]

\noindent
 $\bullet$ 6--8.10.11.13--17; contains $B_3$ formed by 1, 2, 3, 4, 5, 9, 18, 19;\\[-3mm]

\noindent
 $\bullet$ 6--8.10.11.13--15.17; contains $B_1$ formed by 1, 2, 4, 5, 12, 16, 19;\\[-3mm]

\noindent
 $\bullet$ 6--8.10.11.13--15.17.18; contains $B_1$ formed by 1, 2, 4, 5, 12, 16, 19;\\[-3mm]

\noindent
 $\bullet$ 6--8.10.11.13.14.16.17; contains $B_3$ formed by 1, 2, 3, 4, 5, 9, 18, 19;\\[-3mm]

\noindent
 $\bullet$ 6--8.10.11.13.14.16.18; contains $B_3$ formed by 1, 2, 3, 4, 5, 9, 17, 19;\\[-3mm]

\noindent
 $\bullet$ 6--8.10.11.14--17; contains $B_3$ formed by 1, 2, 3, 4, 5, 9, 18, 19;\\[-3mm]

\noindent
 $\bullet$ 6--8.10.11.14.15.17.18; contains $B_1$ formed by 1, 2, 4, 5, 12, 16, 19;\\[-3mm]

\noindent
 $\bullet$ 6--8.10.11.14.16.17; contains $B_3$ formed by 1, 2, 3, 4, 5, 9, 18, 19;\\[-3mm]

\noindent
 $\bullet$ 6--8.10.12--17; contains $B_3$ formed by 1, 2, 3, 4, 5, 9, 18, 19;\\[-3mm]

\noindent
 $\bullet$ 6--8.10.12--16.18; contains $B_1$ formed by 1, 2, 3, 5, 11, 17, 19;\\[-3mm]

\noindent
 $\bullet$ 6--8.10.12--15.18; contains $B_1$ formed by 1, 2, 3, 5, 11, 17, 19;\\[-3mm]

\noindent
 $\bullet$ 6--8.10.12--14.16--18; contains $B_4$ formed by 2, 3, 4, 5, 9, 11, 15, 19;\\[-3mm]

\noindent
 $\bullet$ 6--8.10.12--14.16.17; contains $B_3$ formed by 1, 2, 3, 4, 5, 9, 18, 19;\\[-3mm]

\noindent
 $\bullet$ 6--8.10.12.13.15--17; contains $B_3$ formed by 1, 2, 3, 4, 5, 9, 18, 19;\\[-3mm]

\noindent
 $\bullet$ 6--8.10.12.13.15.17; contains $B_1$ formed by 1, 2, 4, 5, 14, 16, 18;\\[-3mm]

\noindent
 $\bullet$ 6--8.10.12.13.16--18; contains $B_4$ formed by 1, 2, 3, 4, 9, 14, 15, 19;\\[-3mm]

\noindent
 $\bullet$ 6--8.10.12.15--17; contains $B_1$ formed by 2, 3, 4, 5, 13, 14, 18;\\[-3mm]

\noindent
 $\bullet$ 6--8.12--17; contains $B_1$ formed by 1, 2, 3, 5, 10, 11, 19;\\[-3mm]

\noindent
 $\bullet$ 6--8.12--18; contains $B_1$ formed by 1, 2, 3, 5, 10, 11, 19;\\[-3mm]

\noindent
 $\bullet$ 6--8.12--14.16.17; contains $B_1$ formed by 1, 2, 3, 5, 10, 11, 19;\\[-3mm]

\noindent
 $\bullet$ 6--8.12--14.16--18; contains $B_1$ formed by 1, 2, 3, 5, 10, 11, 19;\\[-3mm]

\noindent
 $\bullet$ 6.7.10--17; contains $B_3$ formed by 1, 2, 3, 4, 5, 8, 18, 19;\\[-3mm]

\noindent
 $\bullet$ 6.7.10--16.18; contains $B_3$ formed by 1, 2, 3, 4, 5, 9, 17, 19;\\[-3mm]

\noindent
 $\bullet$ 6.7.10--15.18; contains $B_1$ formed by 1, 3, 4, 5, 16, 17, 19;\\[-3mm]

\noindent
 $\bullet$ 6.7.10--14.16--18; contains $B_3$ formed by 1, 2, 3, 4, 8, 9, 15, 19;\\[-3mm]

\noindent
 $\bullet$ 6.7.10--14.16.18; contains $B_3$ formed by 1, 2, 3, 4, 5, 9, 17, 19;\\[-3mm]

\noindent
 $\bullet$ 6.7.10--13.15--17; contains $B_3$ formed by 1, 2, 3, 4, 5, 8, 18, 19;\\[-3mm]

\noindent
 $\bullet$ 6.7.10--13.15.17.18; contains $B_3$ formed by 1, 2, 3, 4, 5, 8, 16, 19;\\[-3mm]

\noindent
 $\bullet$ 6.7.10--13.16--18; contains $B_3$ formed by 1, 2, 3, 4, 8, 9, 15, 19;\\[-3mm]

\noindent
 $\bullet$ 6.7.10--13.16.18; contains $B_3$ formed by 1, 2, 3, 4, 5, 9, 17, 19;\\[-3mm]

\noindent
 $\bullet$ 6.7.10--12.15--17; contains $B_1$ formed by 2, 3, 4, 5, 13, 14, 18;\\[-3mm]

\noindent
 $\bullet$ 6.7.10.11.13--17; contains $B_3$ formed by 1, 2, 3, 4, 5, 8, 18, 19;\\[-3mm]

\noindent
 $\bullet$ 6.7.10.11.13--16.18; contains $B_3$ formed by 1, 2, 3, 4, 5, 9, 17, 19;\\[-3mm]

\noindent
 $\bullet$ 6.7.10.11.13--15.17.18; contains $B_1$ formed by 1, 2, 4, 5, 12, 16, 19;\\[-3mm]

\noindent
 $\bullet$ 6.7.10.11.13--15.18; contains $B_1$ formed by 1, 2, 4, 5, 12, 16, 19;\\[-3mm]

\noindent
 $\bullet$ 6.7.10.11.13.14.16--18; contains $B_3$ formed by 1, 2, 3, 4, 8, 9, 15, 19;\\[-3mm]

\noindent
 $\bullet$ 6.7.10.11.13.14.16.18; contains $B_3$ formed by 1, 2, 3, 4, 5, 9, 17, 19;\\[-3mm]

\noindent
 $\bullet$ 6.7.10.11.13.15--17; contains $B_3$ formed by 1, 2, 3, 4, 5, 8, 18, 19; \\[-3mm]

\noindent
 $\bullet$ 6.7.10.11.14.15--17; contains $B_3$ formed by 1, 2, 3, 4, 5, 8, 18, 19;\\[-3mm]

\noindent
 $\bullet$ 6.7.10.11.14--16.18; contains $B_3$ formed by 1, 2, 3, 4, 5, 9, 17, 19;\\[-3mm]

\noindent
 $\bullet$ 6.7.10.11.14.15.17.18; contains $B_1$ formed by 1, 2, 4, 5, 12, 16, 19;\\[-3mm]

\noindent
 $\bullet$ 6.7.10.11.14.15.18; contains $B_1$ formed by 1, 2, 4, 5, 12, 16, 19;\\[-3mm]

\noindent
 $\bullet$ 6.7.10.11.14.16--18; contains $B_3$ formed by 1, 2, 3, 4, 8, 9, 15, 19;\\[-3mm]

\noindent
 $\bullet$ 6.7.10.11.15--17; contains $B_1$ formed by 2, 3, 4, 5, 13, 14, 18;\\[-3mm]

\noindent
 $\bullet$ 6.7.11--18; contains $B_2$ formed by 2, 3, 4, 8, 9, 10, 19;\\[-3mm]

\noindent
 $\bullet$ 6.7.11--14.16--18; contains $B_2$ formed by 2, 3, 4, 8, 9, 10, 19;\\[-3mm]

\noindent
 $\bullet$ 6.7.11--14.16.18; contains $B_2$ formed by 2, 3, 4, 8, 9, 10, 19;\\[-3mm]

\noindent
 $\bullet$ 6.7.11--13.16--18; contains $B_2$ formed by 2, 3, 4, 8, 9, 10, 19;\\[-3mm]

\noindent
 $\bullet$ 6.7.11.13--18; contains $B_1$ formed by 1, 2, 4, 5, 10, 12, 19;\\[-3mm]

\noindent
 $\bullet$ 6.7.11.13--15.17.18; contains $B_1$ formed by 1, 2, 4, 5, 10, 12, 19;\\[-3mm]

\noindent 
$\bullet$ 6.7.11.13.14.16--18; contains $B_1$ formed by 1, 2, 4, 5, 10, 12, 19;\\[-3mm]

\noindent
 $\bullet$ 6.7.11.13.14.16.18; contains $B_1$ formed by 1, 2, 4, 5, 10, 12, 19;\\[-3mm]

\noindent
 $\bullet$ 6.7.11.14--18; contains $B_1$ formed by 1, 2, 4, 5, 10, 12, 19;\\[-3mm]

\noindent
 $\bullet$ 6.10--17; contains $B_2$ formed by 2, 3, 4, 7, 8, 9, 19;\\[-3mm]

\noindent
 $\bullet$ 6.10--18; contains $B_2$ formed by 2, 3, 4, 7, 8, 9, 19;\\[-3mm]

\noindent
 $\bullet$ 6.10--14.16--18; contains $B_2$ formed by 2, 3, 4, 7, 8, 9, 19;\\[-3mm]

\noindent
 $\bullet$ 6.10--14.16.18; contains $B_2$ formed by 2, 3, 4, 7, 8, 9, 19;\\[-3mm]

\noindent
 $\bullet$ 6.10--13.16--18; contains $B_2$ formed by 2, 3, 4, 7, 8, 9, 19;\\[-3mm]

\noindent
 $\bullet$ 6.10.11.13--17; contains $B_2$ formed by 2, 3, 4, 7, 8, 9, 19;\\[-3mm]

\noindent
 $\bullet$ 6.10.11.13--18; contains $B_2$ formed by 2, 3, 4, 7, 8, 9, 19;\\[-3mm]

\noindent
 $\bullet$ 6.10.11.13.14.16.18; contains $B_2$ formed by 2, 3, 4, 7, 8, 9, 19;\\[-3mm]

\noindent
 $\bullet$ 6.10.11.13.14.16--18; contains $B_2$ formed by 2, 3, 4, 7, 8, 9, 19;\\[-3mm]

\noindent
 $\bullet$ 6.10.11.14--17; contains $B_2$ formed by 2, 3, 4, 7, 8, 9, 19;\\[-3mm]

\noindent
 $\bullet$ 6.10.11.14--18; contains $B_2$ formed by 2, 3, 4, 7, 8, 9, 19;\\[-3mm]

\noindent
 $\bullet$ 6.10.11.14--16.18; contains $B_2$ formed by 2, 3, 4, 7, 8, 9, 19;\\[-3mm]

\noindent
 $\bullet$ 6.10.11.14.16--18; contains $B_2$ formed by 2, 3, 4, 7, 8, 9, 19;\\[-3mm]

\noindent
 $\bullet$ 6.10.13--17; contains $B_1$ formed by 1, 3, 4, 5, 11, 12, 19;\\[-3mm]

\noindent
 $\bullet$ 6.10.13--18; contains $B_1$ formed by 1, 3, 4, 5, 11, 12, 19;\\[-3mm]

\noindent
 $\bullet$ 6.10.13.15--17; contains $B_1$ formed by 1, 3, 4, 5, 11, 12, 19;\\[-3mm]
 
\noindent
 $\bullet$ 6.10.13.15--18; contains $B_1$ formed by 1, 3, 4, 5, 11, 12, 19.\\[-3mm]

\end{footnotesize}

\end{document}